\documentclass{amsart}

\usepackage{amssymb}
\usepackage{amsthm}
 \newtheorem{theorem}{Theorem}[section]
 \newtheorem{corollary}[theorem]{Corollary}
 \newtheorem{lemma}[theorem]{Lemma}
 \newtheorem{proposition}[theorem]{Proposition}

\newcommand{\C}{\mathbb{C}}
\begin{document}

\title[Simple exceptional groups of Lie type]{\bf Simple exceptional groups of Lie type are determined by  their character degrees }

\author{Hung P. Tong-Viet}
\email{Tong-Viet@ukzn.ac.za}
\address{School of Mathematical Sciences,
University of KwaZulu-Natal\\
Pietermaritzburg 3209, South Africa}

\date{\today}
\thanks{The author is supported by a post-doctoral fellowship from the
University of KwaZulu-Natal}

\begin{abstract} Let $G$ be a finite group. Denote by $\textrm{Irr}(G)$ the
set of all irreducible complex characters of $G.$ Let $\textrm{cd}(G)=\{\chi(1)\;|\;\chi\in
\textrm{Irr}(G)\}$ be the set of all irreducible complex character degrees of $G$ forgetting
multiplicities, and let $\textrm{X}_1(G)$ be the set of all irreducible complex character degrees
of $G$ counting multiplicities. Let $H$ be any non-abelian simple exceptional group of Lie type. In
this paper, we will show that if $S$ is a non-abelian simple group and $\textrm{cd}(S)\subseteq
\textrm{cd}(H)$ then $S$ must be isomorphic to $H.$ As a consequence, we show that if $G$ is a
finite group with $\textrm{X}_1(G)\subseteq \textrm{X}_1(H)$ then $G$ is isomorphic to $H.$ In
particular, this implies that the simple exceptional groups of Lie type are uniquely determined by
the structure of their complex group algebras.
\end{abstract}



\date{\today}

\maketitle
\section{Introduction and Notation}
All groups considered are finite and all characters are complex characters. Let $G$ be a group.
Denote by $\textrm{Irr}(G)$ the set of all irreducible characters of $G.$ Let $\textrm{cd}(G)$ be
the set of all irreducible character degrees of $G$ forgetting multiplicities, that is,
$\textrm{cd}(G)=\{\chi(1)\;|\;\chi\in \textrm{Irr}(G)\},$ and let $\textrm{X}_1(G)$ be the set of
all irreducible character degrees of $G$ counting multiplicities. Observe that $\textrm{X}_1(G)$ is
just the first column of the ordinary character table of $G.$ We follow \cite{atlas} for notation
of non-abelian simple groups. In this paper, we will prove the following results.

\begin{theorem}\label{main} Let $H$ be a simple exceptional group of Lie type.
 If $S$ is a non-abelian simple group and  $\emph{\textrm{cd}}(S)\subseteq \emph{\textrm{cd}}(H),$
then $S\cong H.$
\end{theorem}
\begin{corollary}\label{main1} Let $H$ be a simple exceptional group of Lie type.
If  $G$ is a perfect group, $\emph{\textrm{cd}(G)}\subseteq \emph{\textrm{cd}}(H)$ and $|G|\leq
|H|,$ then $G\cong H.$
\end{corollary}

\begin{corollary}\label{main2} Let $H$ be a simple exceptional group of Lie type. If $G$ is a group
and $\emph{\textrm{X}}_1(G)\subseteq \emph{\textrm{X}}_1(H),$ then $G\cong H.$
\end{corollary}
Corollary \ref{main2} gives a positive answer to \cite[Question $11.8(a)$]{Mazurov} for simple
exceptional groups of Lie type. For alternating groups of degree at least $5,$ sporadic simple
groups, and the Tits group,  similar results were obtained in \cite{Hung1}. Corollary \ref{main1}
is related to the Huppert's Conjecture \cite{Hupp} which says that the non-abelian simple groups
are determined by the set of their character degrees. Let $\C$ be the complex number field and let
$G$ be a group. Denote by $\C G$ the group algebra of $G$ over $\C.$

\begin{corollary}\label{main3} Let $H$ be a simple exceptional group of Lie type. If $G$ is a group
such that $\C G\cong \C H,$ then $G\cong H.$
\end{corollary}
Basically, this corollary says that if $H$ is a simple exceptional group of Lie type  then $H$ is
uniquely determined by the structure of its complex group algebra $\C H.$ This result is related to
the Brauer's Problem $2$ which asks the following question: Which groups can be determined by the
structure of their complex group algebras? In \cite{Hung1} and \cite{Hung3}, we have shown that the
alternating groups of degree at least $5,$ the sporadic simple groups  and the symmetric groups are
uniquely determined by the structure of their complex group algebras. In \cite{Hung2}, we will
establish a similar result for the remaining simple classical groups of Lie type.

We fix some notation. If $\textrm{cd}(G)=\{s_0,s_1,\cdots,s_t\},$ where $1=s_0<s_1<\cdots<s_t,$
then we define $d_i(G)=s_i$ for all $1\leq i\leq t.$  Then $d_i(G)$ is the $i^{th}$ smallest degree
of the non-trivial character degrees of $G.$ We also define $t(G)=t+1,$ the number of distinct
character degrees of $G.$ We put $b(G)=s_t.$ Then $b(G)$ is the largest character degree of $G.$ If
$n$ is an integer then we denote by $\pi(n)$ the set of all prime divisors of $n.$ If $G$ is a
group, we will write $\pi(G)$ instead of $\pi(|G|)$ to denote the set of all prime divisors of the
order of $G.$
\section{Preliminaries}
The following result is a well-known theorem due to Zsigmondy.
\begin{lemma}\emph{(\cite[Theorems $5.2.14,5.2.15$]{KL}).}
Let $q$ and $n$ be integers with $q\geq 2$ and $n\geq 3.$ Then
$q^n-1$ has a prime $l$ such that $l$ does not divide $q^m-1$ for
$m<n,$ unless $(q,n)=(2,6).$ Moreover if  $l\mid q^k-1$ then $n\mid
k.$
\end{lemma}
Such an $l$ is called a \emph{primitive prime divisor}. We denote by
$l_n(q)$ the smallest primitive prime divisor of $q^n-1$ for fixed
$q$ and $n.$ When $n$ is odd and $(q,n)\neq (2,3)$ then there is a
primitive prime divisor of $q^{2n}-1$ and we denote this primitive
prime divisor by $l_{-n}(q).$

 The next
result gives a lower bound for the largest character degree of the alternating groups. Recall that
$k(G)$ is the number of conjugacy classes of $G.$
\begin{lemma}\label{Lem3} Assume that $n\geq 10.$ Then $b(A_n)\geq 2^{n-1}.$
\end{lemma}
\begin{proof} Using \cite{GAP}, we can check that $b(A_n)\geq 2^{n-1}$
for all $10\leq n\leq 17.$ Thus we can assume  $n\geq 18.$ Observe that as $A_n\unlhd S_n,$ we
obtain $b(S_n)\leq |S_n:A_n|b(A_n)=2b(A_n)$ (see \cite[Exercise $5.4,$ p. $74$]{Isaacs}). Thus it
suffices to show that $b(S_n)\geq 2^n,$ for any $n\geq 18.$ We have $|S_n|=n!=\sum_{\chi\in
\textrm{Irr}(S_n)}\chi(1)^2\leq k(S_n)b(S_n)^2.$ It follows from \cite[Theorem $1.1$]{Mar} that
$k(S_n)\leq 3^{(n-1)/2}$ and hence $b(S_n)^2\geq n!/k(S_n)\geq n!/3^{(n-1)/2}.$ Thus we only need
to verify that $n!\geq 3^{(n-1)/2}\cdot 2^{2n},$ for all $n\geq 18.$ We will prove this inequality
by induction on $n\geq 18.$ Obviously this is true for $n=18.$ Assume that $m!\geq 3^{(m-1)/2}\cdot
2^{2m},$ for some $m\geq 18.$ We need to show that $(m+1)!\geq 3^{m/2}\cdot 2^{2m+2}.$ By induction
hypothesis, we have that $(m+1)!=(m+1)\cdot m!\geq (m+1)\cdot 3^{(m-1)/2}\cdot 4^m.$ Thus, we need
to prove that $m+1\geq 4\cdot 3^{1/2}.$ As $4\cdot 3^{1/2}<4\cdot 2=8<18<m+1,$ we obtain
$(m+1)!\geq 3^{m/2}\cdot 2^{2m+2},$ and so $n!\geq 3^{(n-1)/2}\cdot 2^{2n},$ for any $n\geq 18.$ We
have shown that $b(S_n)^2\geq 2^{2n},$ which implies $b(S_n)\geq 2^n.$ This finishes the proof.
\end{proof}


The following lemma gives some basic properties of simple groups which satisfy the hypotheses of
Theorem \ref{main}. These results will be used frequently in the next section. The conclusions
$(ii)$ and $(iii)$ are obvious,  while $(i)$ and $(iv)$ can be found in \cite[Lemma $11$]{Hung1}.
\begin{lemma}\label{Lem4} Let $S$ and $H$ be non-abelian simple groups.
If $\emph{\textrm{cd}}(S)\subseteq \emph{\textrm{cd}}(H)$ then the following holds:

$(i)$ $d_i(S)\geq d_i(H),$ for all $i;$

$(ii)$ $b(S)\leq b(H);$

$(iii)$ $t(S)\leq t(H);$

$(iv)$ $\pi(S)\subseteq \pi(H).$
\end{lemma}

The next two results are well-known (see \cite[Theorems $3.10,3.11$]{Isaacs}), we will use them
freely without further reference.
\begin{lemma} Let $G$ be a group. Then the following holds.

$(a)$ If $|\pi(G)|\leq 2,$ then $G$ is solvable.

$(b)$ If $\chi\in \emph{\textrm{Irr}}(G)$ then $\chi(1)\mid |G|.$

\end{lemma}

In Table \ref{Ta1} we list the orders of the simple exceptional groups of Lie type.  Table
\ref{Ta2} is taken from \cite[Table $5.3.A$]{KL} which gives the Landazuri-Seitz-Zalesskii bounds
for non-trivial minimal degrees of the cross-characteristic representations of finite simple groups
of Lie type. In Table \ref{Ta3}, we give the $p$-part of the degree of a unipotent character $\psi$
of the simple groups of Lie type, where $\psi$ is not the Steinberg character. We refer to
\cite{car85} for the classification of unipotent characters and the notion of symbols. In Table
\ref{Ta4}, we have $b(Fi_{24}')=336033532800,$ $b(B)=16547812226400000$ and
$b(M)=258823477531055064045234375.$ The results in this table are taken from \cite{atlas}. Finally,
in Table \ref{Ta5} we list the upper bounds for the largest character degree of  the simple
exceptional group  of Lie type. These upper bounds were obtained in \cite[Theorem $2.1$]{Seitz}.

\begin{table}
 \begin{center}
  \caption{The orders of simple exceptional groups} \label{Ta1}
  \begin{tabular}{c|c|c}
   \hline
   $S$  &$d$& $|S|$\\ \hline
   ${}^2B_2(q^2),q^2=2^{2m+1}$&$1$&$q^4(q^4+1)(q^2-1)$\\
   ${}^2G_2(q^2),q^2=3^{2m+1}$&$1$&$q^6(q^6+1)(q^2-1)$\\
   ${}^2F_4(q^2),q^2=2^{2m+1}$&$1$&$q^{24}(q^{12}+1)(q^8-1)(q^6+1)(q^2-1)$\\
   ${}^3D_4(q)$&$1$&$q^{12}(q^8+q^4+1)(q^6-1)(q^2-1)$\\
   ${}^2E_6(q)$&$(3,q+1)$&$q^{36}\prod_{i\in\{2,5,6,8,9,12\}}(q^i-(-1)^i)/d$\\
   $G_2(q)$&$1$&$q^{6}(q^6-1)(q^2-1)$\\
   $F_4(q)$&$1$&$q^{24}\prod_{i\in\{2,6,8,12\}}(q^i-1)$\\
   $E_6(q)$&$(3,q-1)$&$q^{36}\prod_{i\in\{2,5,6,8,9,12\}}(q^i-1)/d$\\
   $E_7(q)$&$(2,q-1)$&$q^{63}\prod_{i\in\{2,6,8,10,12,14,18\}}(q^i-1)/d$\\
   $E_8(q)$&$1$&$q^{120}\prod_{i\in\{ 2,8,12,14,18,20,24,30\}}(q^i-1)$\\\hline
  \end{tabular}
 \end{center}
\end{table}

\begin{table}
 \begin{center}
  \caption{The Landazuri-Seitz-Zalesskii bounds} \label{Ta2}
  \begin{tabular}{c|c|c}
   \hline
   $S$  & $e(S)$& Exceptions\\ \hline
   ${}^2B_2(q^2),q^2=2^{2m+1}$&$q(q^2-1)/\sqrt{2}$&${}^2B_2(8)$\\
   ${}^2G_2(q^2),q^2=3^{2m+1}$&$q^2(q^2-1)$&\\
   ${}^2F_4(q^2),q^2=2^{2m+1}$&$q^9(q^2-1)/\sqrt{2}$&\\
   ${}^3D_4(q)$&$q^3(q^2-1)$&\\
   ${}^2E_6(q)$&$q^9(q^2-1)$&\\
   $G_2(q)$&$q(q^2-1)$&$G_2(3),G_2(4)$\\
   $F_4(q)$&$q^6(q^2-1),\mbox{$q$ odd}$&\\
   &$q^7(q^3-1)(q-1)/2,\mbox{$q$ even}$&$F_4(2)$\\
   $E_6(q)$&$q^9(q^2-1)$&\\
   $E_7(q)$&$q^{15}(q^2-1)$&\\
   $E_8(q)$&$q^{27}(q^2-1)$&\\\hline
  \end{tabular}
 \end{center}
\end{table}

\begin{table}
 \begin{center}
  \caption{Some unipotent characters of simple groups of Lie type} \label{Ta3}
  \begin{tabular}{l|l|r}
   \hline
   $S=S(p^b)$  & Symbol &$p$-part of degree\\ \hline
   $L_n^\epsilon(p^b)$ & $(1^{n-2},2)$&$p^{b(n-1)(n-2)/2}$\\
   $S_{2n}(p^b),p=2$&$\binom{0\:1\:2\:\cdots\:n-2\:n-1\:n}{\:\:1\:2\cdots\:n-2}$&$2^{b(n-1)^2-1}$\\
   $S_{2n}(p^b),p>2$ && $p^{b(n-1)^2}$\\
   $O_{2n+1}(p^b),p>2$  &$\binom{0\:1\:2\:\cdots\:n-2\:n-1\:n}{\:\:1\:2\cdots\:n-2}$& $p^{b(n-1)^2}$\\
   $O_{2n}^+(p^b)$&$\binom{0\:1\:2\:\cdots\:n-3\:n-1}{\:1\:2\:3\cdots\:n-2\:n-1}$&$p^{b(n^2-3n+3)}$\\
   $O_{2n}^-(p^b)$&$\binom{0\:1\:2\:\cdots\:n-2\:n-1}{\:\:1\:2\cdots\:n-2}$&$p^{b(n^2-3n+2)}$\\
   ${}^3D_4(p^b)$&$\phi_{1,3}''$&$p^{7b}$\\
   $F_4(p^b)$&$\phi_{9,10}$&$p^{10b}$\\
   ${}^2F_4(q^2)$&${}^2B_2[a],\epsilon$&$\frac{1}{\sqrt{2}}q^{13}$\\
   $E_6(p^b)$&$\phi_{6,25}$&$p^{25b}$\\
   ${}^2E_6(p^b)$&$\phi_{2,16}''$&$p^{25b}$\\
   $E_7(p^b)$&$\phi_{7,46}$&$p^{46b}$\\
   $E_8(p^b)$&$\phi_{8,91}$&$p^{91b}$\\\hline
\end{tabular}
\end{center}
\end{table}

\begin{table}
 \begin{center}
  \caption{Sporadic simple groups}\label{Ta4}
  \begin{tabular}{l|c|c|c|c|r}
   \hline
   $S$  & $t(S)$&$d_1(S)$&$d_2(S)$&$d_3(S)$&$b(S)$\\ \hline
   $M_{11}$&$7$&$10$&$11$&$16$&$55$\\
   $M_{12}$&$11$&$11$&$16$&$45$&$176$\\
   $J_1$&$7$&$56$&$76$&$77$&$209$\\
   $M_{22}$&$10$&$21$&$45$&$55$&$385$\\
   $J_2$&$16$&$14$&$21$&$36$&$336$\\
   $M_{23}$&$11$&$22$&$45$&$230$&$2024$\\
   $HS$&$18$&$22$&$77$&$154$&$3200$\\
   $J_3$&$14$&$85$&$323$&$324$&$3078$\\
   $M_{24}$&$20$&$23$&$45$&$231$&$10395$\\
   $McL$&$17$&$22$&$231$&$252$&$10395$\\
   $He$&$23$&$51$&$153$&$680$&$23324$\\
   $Ru$&$28$&$378$&$406$&$783$&$118784$\\
   $Suz$&$36$&$143$&$364$&$780$&$248832$\\
   $O'N$&$19$&$10944$&$13376$&$25916$&$234080$\\
   $Co_3$&$34$&$23$&$253$&$275$&$255024$\\
   $Co_2$&$51$&$23$&$253$&$275$&$2095875$\\
   $Fi_{22}$&$52$&$78$&$429$&$1001$&$2729376$\\
   $HN$&$42$&$133$&$760$&$3344$&$5878125$\\
   $Ly$&$36$&$2480$&$45694$&$48174$&$71008476$\\
   $Th$&$39$&$248$&$4123$&$27000$&$190373976$\\
   $Fi_{23}$&$84$&$782$&$3588$&$5083$&$559458900$\\
   $Co_1$&$98$&$276$&$299$&$1771$&$551675124$\\
   $J_4$&$42$&$1333$&$299367$&$887778$&$3054840657$\\
   $Fi_{24}'$&$91$&$8671$&$57477$&$249458$&$b(Fi_{24}')$\\
   $B$&$165$&$4371$&$96255$&$1139374$&$b(B)$\\
   $M$&$170$&$196883$&$21296876$&$842609326$&$b(M)$\\
   ${}^2F_4(2)'$&$14$&$26$&$27$&$78$&$2048$\\\hline
  \end{tabular}
 \end{center}
\end{table}

\begin{table}
 \begin{center}
  \caption{Upper bounds of $b(S)$ for simple groups of Lie type}\label{Ta5}
  \begin{tabular}{c|c|c|c|c|c}
   \hline
   $S$  &         $F_4(q)$ & $G_2(q)$  &${}^2B_2(q^2)$&${}^2F_4(q^2)$ & ${}^2G_2(q^2)$\\ \hline
   $\tilde{b}(S)$&$q^{28}$ &$q^8$      &$q^6$&$q^{28}$ &$q^8$\\\hline
    $S$  &   $E_6(q)$&$E_7(q)$&$E_{8}(q)$ &${}^2E_{6}(q)$  &${}^3D_{4}(q)$  \\ \hline
   $\tilde{b}(S)$&$q^{42}$&$q^{70}$  &$q^{128}$& $q^{42}$&$q^{17}$\\\hline
  \end{tabular}
 \end{center}
\end{table}

\section{Proof of the main results}
We assume the following set up. Let $\mathbb{G}$ be a simple simply connected algebraic group of
exceptional type defined over $\bar{\mathbb{F}}_p,$ where $p$ is prime. Let $L$ be the fixed point
group of $\mathbb{G}$ under a suitable Frobenius morphism such that $L/Z(L)$ is isomorphic to the
simple exceptional group of Lie type $H.$ Generically $L$ is the universal covering group of $H.$
We remark that if $\textrm{cd}(S)\subseteq \textrm{cd}(H)$ then $\textrm{cd}(S)\subseteq
\textrm{cd}(L)$ for any group $S.$ Using the classification of finite simple groups, $S$ is a
sporadic group, an alternating group of degree at least $5$ or a simple group of Lie type. We will
treat the Tits group as a sporadic group. We will prove Theorem \ref{main} by a series of
propositions. Using \cite[Theorem \textrm{C}]{Malle05},  the result for Suzuki groups follows
easily.


\begin{proposition}\label{prop1} If $S$ is a non-abelian simple
group with $\emph{\textrm{cd}}(S)\subseteq \emph{\textrm{cd}}({}^2B_2(q^2)),$ where
$q^2=2^{2m+1},m\geq 1,$ then $S\cong {}^2B_2(q^2).$
\end{proposition}

\begin{proof} Let $H\cong {}^2B_2(2^{2m+1}).$ Assume that $S$ is a
non-abelian simple group with $\textrm{cd}(S)\subseteq \textrm{cd}(H),$ we will show that $S\cong
H.$ By Lemma \ref{Lem4}$(iii)$ and \cite[Theorem \textrm{C}]{Malle05}, we have  $4\leq t(S)\leq
t(H)= 6$ and $S\in \{L_2(r),L_3(4),{}^2B_2(2^{2n+1})\},$ where $r$ is a prime power. If $S=L_2(r)$
or $S=L_3(4)$ then $3\in \pi(S).$ However it is well-known that $3\not\in \pi(H),$ which
contradicts Lemma \ref{Lem4}$(iv).$ Hence $S={}^2B_2(2^{2n+1}).$ By \cite[Theorem $1.1$]{Malle},
the only nontrivial character degree of $H$ which is a power of $2$ is the degree of the Steinberg
character of $H$ of degree $|H|_2.$ As $S$ also possesses a character of degree $|S|_2,$ it follows
that $|S|_2=|H|_2$ so that $2(2n+1)=2(2m+1),$ which implies that $m=n.$ Thus $S\cong H$ as
required.
\end{proof}

Therefore we can assume that $H$ is a simple exceptional group of Lie type defined over a field of
size $q$  in characteristic $p$ with  $H\neq {}^2B_2(q^2),{}^2F_4(2)'.$

\begin{proposition}\label{prop2} If $S$ is a sporadic simple group or the Tits group, then
$\emph{\textrm{cd}}(S)\nsubseteq \emph{\textrm{cd}}(H).$
\end{proposition}

\begin{proof}
Let $S$ be a simple sporadic group or the Tits group. By way of contradiction, assume that
$\textrm{cd}(S)\subseteq \textrm{cd}(H).$  We outline our general argument here. By Lemma
\ref{Lem4}$(i)$, we have $d_1(S)\geq d_1(H)$ and hence by Table \ref{Ta2}, we obtain, apart from
some exceptions, $d_1(S)\geq e(H),$ where $e(H)$ is the Landazuri-Seitz-Zalesskii bound for $H.$
Solving this inequality, we get an upper bound for $q.$  For these values of $q,$ we will show that
$\pi(S)\not\subseteq \pi(H)$ and hence $\textrm{cd}(S)\not\subseteq \textrm{cd}(H)$ by Lemma
\ref{Lem4}$(iv).$

{\bf Case $1.$} $H={}^2G_2(q^2),q^2=3^{2m+1},m\geq 1.$ By Lemma
\ref{Lem4}$(iii)$ and \cite{Ward}, we have $t(S)\leq 11.$ By Table
\ref{Ta4}, we only need to consider the following cases
$\{M_{11},M_{12},J_1,$ $M_{22},$ $M_{23}\}.$ By Lemma \ref{Lem4}$(i),$
we have $d_1(S)\geq d_1(H)$ so that by Table \ref{Ta2} we
obtain $d_1(S)\geq 3^{2m+1}(3^{2m+1}-1).$ As $m\geq 1,$ we
have $3^{2m+1}(3^{2m+1}-1)\geq 702>56\geq d_1(S)$ for any sporadic
simple groups considered above.

{\bf Case $2.$} $H=G_2(q).$ By Lemma \ref{Lem4}$(iii)$ and
\cite{Chang,Eno76,Eno86}, we have $t(S)\leq 23.$ By Table \ref{Ta4},
we only need to consider the following cases
$\{M_{11},M_{12},J_1,M_{22},J_2,M_{23},HS,J_3,$ $M_{24},$
$McL,He,O'N,{}^2F_4(2)'\}.$ Using \cite{atlas}, we can assume that
$q\geq 7.$ Thus by Lemma \ref{Lem4}$(i),$ we have $d_1(S)\geq
d_1(H)$ so that by Table \ref{Ta2} we obtain $d_1(S)\geq q(q^2-1),$
where $q\geq 7.$ Since $q(q^2-1)\geq 336,$ we see that
$d_1(S)<336\leq d_1(H)$ unless $S=O'N.$ We have
$\pi(O'N)=\{2,3,5,7,11,19,31\}.$ As $d_1(S)=10944\geq q(q^2-1),$ we
deduce that $7\leq q\leq 23.$ By Lemma \ref{Lem4}$(iv),$ we have
$\{2,3,5,7,11,19,31\}\subseteq \pi(G_2(q)).$ However we can check
that $\pi(O'N)\not\subseteq \pi(G_2(q))$ for any $7\leq q\leq 23.$

{\bf Case $3.$} $H={}^3D_4(q).$ By Lemma \ref{Lem4}$(iii)$ and \cite{Der}, we have $t(S)\leq 33.$
By Table \ref{Ta4}, we only need to consider the following cases
$\{M_{11},M_{12},J_1,M_{22},J_2,M_{23},HS,J_3,$ $M_{24},$ $McL,He,Ru,O'N,{}^2F_4(2)'\}.$ Using
\cite{atlas}, we can assume that $q\geq 3.$ If $q=3$ then by \cite{Lub}, we have $d_2(H)\geq
3942>d_2(S)$ unless $S=O'N.$ But then  $\pi(O'N)\not\subseteq \pi({}^3D_4(3)),$ which contradicts
Lemma \ref{Lem4}$(iv).$ Thus $q\geq 4.$ By Lemma \ref{Lem4}$(i),$ we have $d_1(S)\geq d_1(H)$ so
that by Table \ref{Ta2} we obtain $d_1(S)\geq q^3(q^2-1).$ As $q\geq 4,$ we have $q^3(q^2-1)\geq
960,$ and so $d_1(S)<960\leq d_1(H)$ for all sporadic simple groups above, unless $S=O'N.$ As
$d_1(O'N)\geq q^3(q^2-1),$ we deduce that $q\leq 5$ so that $4\leq q\leq 5.$ However we can check
that $\pi(O'N)\not\subseteq \pi({}^3D_4(q))$ for any $4\leq q\leq 5,$ which contradicts Lemma
\ref{Lem4}$(iv).$

{\bf Case $4.$} $H={}^2F_4(q^2),q^2=2^{2m+1},m\geq 1.$  By Lemma \ref{Lem4}$(iv)$ and Table
\ref{Ta2}, we have $d_1(S)\geq d_1(H)\geq 2^{9m+4}(2^{2m+1}-1).$ As $m\geq 1,$ we have $d_1(H)\geq
57344>d_1(S)$ unless $S=M.$ If $m\geq 2$ then $d_1(H)\geq 130023424>d_1(M).$ Thus $m=1$ and so
$H={}^2F_4(8).$ However we can check that $\pi(M)\not\subseteq \pi({}^2F_4(8)).$ Thus
$\textrm{cd}(S)\not\subseteq \textrm{cd}(H).$

{\bf Case $5.$} $H={}^2E_6(q)$ or $E_6(q).$  We have $d_1(S)\geq d_1(H)\geq q^9(q^2-1)$ by Table
\ref{Ta2}.  Using \cite{atlas}, we can assume that $q\geq 3.$ If $q\geq 4,$ then  $d_1(H)\geq
3932160>d_1(S)$ for all sporadic groups. Hence $q=3.$ By \cite{Lub}, we have $d_2(H)\geq
32690203>21296876\geq d_2(S)$ for all sporadic groups. Thus $\textrm{cd}(S)\not\subseteq
\textrm{cd}(H).$

{\bf Case $6.$} $H=F_4(q).$ By \cite{atlas}, we can assume that
$q\geq 3.$ Assume first that $q$ is even. Then $q\geq 4$ and by
Table \ref{Ta2}, $d_1(H)\geq q^7(q^3-1)(q-1)/2\geq 1548288>d_1(S)$
for all sporadic groups. Thus $q$ is odd and $q\geq 3.$ By Table
\ref{Ta2}, we have $d_1(H)\geq q^6(q^2-1).$ If $q\geq 5,$ then
$d_1(H)\geq 375000>d_1(S)$ for all $S.$ Hence $q=3.$  By \cite{Lub},
we have either $d_1(H)\geq 6643>d_1(S)$ or $d_2(H)\geq 83148>d_2(S)$ unless
$S=M.$ But then $\pi(M)\not\subseteq \pi(F_4(3)).$

{\bf Case $7.$} $H=E_7(q)$ or $E_8(q).$ Using \cite{Lub}, one can assume that $q\geq 3$ since
$d_2(E_8(2))\geq d_2(E_7(2))>d_2(S)$ for all $S.$ By Table \ref{Ta2}, we have $d_1(H)\geq
q^{15}(q^2-1)\geq 114791256>d_1(S)$ for all $S.$ Thus $\textrm{cd}(S)\not\subseteq \textrm{cd}(H).$
The proof is now complete.

\end{proof}

\begin{proposition}\label{prop3} If $m\geq 5,$ then  $\emph{\textrm{cd}}(A_m)\nsubseteq \emph{\textrm{cd}}(H).$
\end{proposition}
\begin{proof}
By way of contradiction, suppose that $\textrm{cd}(A_m)\subseteq \textrm{cd}(H).$ Using
\cite{atlas} and Table \ref{Ta2}, we see that $d_1(H)\geq 10$ so that by Lemma \ref{Lem4}$(i),$ we
can assume that $m\geq 11.$ Assume that $H=G_2(3).$ Then $d_1(H)=14$ and $11\not\in \pi(H).$ As
$m\geq 11,$ we have $d_1(A_m)=m-1\geq d_1(H)=14$ so that $m\geq 15.$ But then $11\in \pi(A_m),$
which contradicts Lemma \ref{Lem4}$(iv).$ Using the same argument, we can assume $H\neq
G_2(3),G_2(4),F_4(2).$ By Lemma \ref{Lem4}$(i,ii)$ and Table \ref{Ta2}, we have $d_1(A_m)=m-1\geq
d_1(H)\geq e(H)$ and $b(A_m)\leq b(H).$ Combining with Lemma \ref{Lem3}, we obtain $b(H)\geq
2^{e(H)}.$ Now by Table \ref{Ta5}, we have $\tilde{b}(H)\geq b(H)\geq 2^{e(H)}.$ However it is
routine to check that this inequality cannot happen so that $\textrm{cd}(A_m)\not\subseteq
\textrm{cd}(H).$

\end{proof}

\begin{proposition}\label{prop4} If $S$ is a simple group of Lie type
in cross-characteristic then  $\emph{\textrm{cd}}(S)\not\subseteq \emph{\textrm{cd}}(H).$
\end{proposition}
\begin{proof} Suppose that $\textrm{cd}(S)\subseteq \textrm{cd}(H)$ and that $S$ is a
simple group of Lie type in characteristic $s\neq p.$ Then $S$ possesses a character of degree
$|S|_s$ which is a nontrivial power of $s.$ As $\textrm{cd}(S)\subseteq \textrm{cd}(H),$ we deduce
that $H$ has two different nontrivial prime power degrees. By \cite[Theorem $1.1$]{Malle}, we
deduce that $H=G_2(3)$ and the only nontrivial prime power degrees in $H$ are $2^6$ and $3^6.$ As
$\pi(H)=\{2,3,7,13\}$ and the characteristic of $H$ is $3,$ we deduce that $S$ is a simple group of
Lie type in characteristic $2$ with $3\leq |\pi(S)|\leq 4,$ $\pi(S)\subseteq \{2,3,7,13\}$ and
$|S|_2=2^6.$ However, by checking the list of simple groups with $3$ or $4$ prime divisors in
\cite[Lemmas $2,3$]{Hag}, there is no simple groups satisfying these conditions. (Note that  in
\cite[Lemma $3$]{Hag} the author missed the group $U_4(3)$). This contradiction proves the
proposition.
\end{proof}

\begin{proposition}\label{prop5} If $S$ is a simple group of Lie type of characteristic $p$ as that of
$H,$ and $\emph{\textrm{cd}}(S)\subseteq \emph{\textrm{cd}}(H),$ then $S\cong H.$
\end{proposition}
\begin{proof} Assume that $S$ is a simple group of Lie type in
characteristic $p,$ defined over a field of size $r=p^b$ and $\textrm{cd}(S)\subseteq
\textrm{cd}(H).$

Observe first that by \cite[Theorem $1.1$]{Malle}, the only
nontrivial character degree of $H$ which is a power of $p$ is the
degree of the Steinberg character $St_H$ of $H$ of degree $|H|_p.$
As $S$ is also a simple group of Lie type in characteristic $p,$ it
follows that $|S|_p=|H|_p.$ We note that the upper bounds for the
$p$-part of the non-Steinberg characters of $F_4(q),
{}^2E_6(q),E_6(q),E_7(q)$ and $E_8(q)$ are taken from \cite{Lubweb}.

{\bf Case $1.$} $H\cong {}^2G_2(q^2),q^2=3^{2m+1},m\geq 1.$ We have $p=3$ and $|H|_p=3^{3(2m+1)}.$
The character degrees of $H$ are available in \cite{Ward}. Observe that if $\chi\in
\textrm{Irr}(H)$ with $3\mid \chi(1)$ and $\chi\neq St_H,$ then $\chi(1)_3\in \{3^m,3^{2m+1}\}.$

$(a)$ $S=L_n^\epsilon(r),r=p^b.$ We have $p=3$ and $bn(n-1)=6(2m+1).$ If $n=2$ then $r=q^6$ so that
$r-1=q^6-1\in \textrm{cd}(S)\subseteq \textrm{cd}({}^2G_2(q^2)),$ which is impossible. If $n=3$
then $r=q^2.$ If $S=L_3(q^2),$ then $r^3-1=q^6-1\in \textrm{cd}(S)\subseteq
\textrm{cd}({}^2G_2(q^2)),$ a contradiction. If $S=U_3(q^2),$ then it possesses an irreducible
character of degree $r(r-1)=q^2(q^2-1),$ hence $q^2(q^2-1)\in \textrm{cd}({}^2G_2(q^2)),$ a
contradiction. Hence $n\geq 4.$ By \cite[$13.8$]{car85}, $S$ possesses unipotent characters of
degrees $(r^n-\epsilon^{n-1}r)/(r-\epsilon)$ and
$r^2(r^{n}-\epsilon^n)(r^{n-3}-\epsilon^{n-3})/((r-\epsilon)(r^2-1))$ labeled by the symbols
$(1,n-1)$ and $(2,n-2),$ respectively. Observe that these degrees are not $p$-powers, so that
$r=3^m$ and $r^2=3^{2m+1}.$ As $r^2=3^{2m+1}=3\cdot (3^m)^2=3r^2,$ we have $1=3,$ a contradiction.

$(b)$ $S=S_{2n}(r),n\geq 2,S\neq S_4(2)$ or $S=O_{2n+1}(r),n\geq 3,r$ odd. We have $p=3$ and
$bn^2=3(2m+1).$ It follows that $n$ must be odd and so $n\geq 3.$ By \cite[$13.8$]{car85}, $S$
possesses a unipotent character labeled by the symbol $\binom{0\:1\:n}{\:-\:}$ with degree
$r(r^n-1)(r^{n-1}-1)/2(r+1).$ By Table \ref{Ta3}, $S$ also possesses a unipotent character $\psi$
with $\psi(1)_p=r^{(n-1)^2}.$ As $r$ is odd, we have $r=3^m$ and since $n\geq 3,$
$r^{(n-1)^2}=3^{2m+1}.$ As $r=3^m,$ we have $b=m$ and so $mn^2=3(2m+1).$ It follows that $m\mid 3$
so that $m=1$ or $m=3.$ If $m=1$ then $n=3$ but then $r^{(n-1)^2}=3^4\neq 3^3=3^{2m+1}.$ Thus $m=3$
and so $n^2=7,$ which is impossible.

$(c)$ $S=O_{2n}^\epsilon(r),n\geq 4.$ We have $p=3$ and
$bn(n-1)=3(2m+1).$ As $n\geq 4,$ $bn(n-1)$ is even but
$3(2m+1)$ is always odd. Hence this case cannot happen.

$(d)$ $S={}^2B_2(2^{2n+1})$ or $S={}^2F_4(2^{2n+1}).$ These cases
cannot happen since $S$ is of characteristic $2.$

$(e)$ $S={}^2G_2(3^{2n+1}).$ Since $3(2n+1)=3(2m+1),$ we have $m=n,$
 hence $S\cong H.$

$(f)$ $S\in \{{}^3D_4(r),G_2(r),F_4(r),E_6(r),{}^2E_6(r),E_8(r)\}.$
In these cases, we see that $|S|_3$ is an even power of $3$ while
$|H|_3=3^{3(2m+1)}.$ Thus these cases cannot happen.

$(g)$ $S=E_7(r).$ We have $63b=3(2m+1)$ hence $21b=2m+1$ and so $q^2=r^{21}.$ By
\cite[$13.9$]{car85}, $S$ possesses unipotent characters of degrees
$\phi_{7,1}(1)=r\Phi_7\Phi_{12}\Phi_{14}$ and
$\phi_{27,2}(1)=r^2\Phi_3^2\Phi_6^2\Phi_9\Phi_{12}\Phi_{18}.$ Hence $r=3^m$ and $r^2=3^{2m+1},$
which is impossible.

{\bf Case $2.$} $H={}^2F_4(q^2),q^2=2^{2m+1},m\geq 1.$ We have $p=2$ and $|H|_2=2^{12(2m+1)}.$
Using the list of character degrees of ${}^2F_4(q^2)$ in \cite{Lubweb}, we see that if $\chi\in
\textrm{Irr}(H)$ with $\chi(1)$ is even, then either $\chi(1)_2=2^m$ or $2^{2m+1}\mid \chi(1).$
Moreover $\chi(1)_2\leq 2^{13m+6}<q^{13}$ if $\chi\neq St_H.$

$(a)$ $S=L_n^\epsilon(r),n\geq 2.$ Then $bn(n-1)=24(2m+1).$ If $n=2,$ then $b=12(2m+1)$ so
$r=q^{24}$ and $S=L_2(q^{24}).$ We have $r+1=q^{24}+1\in \textrm{cd}(S)$ but $q^{24}+1\not\in
\textrm{cd}(H),$ a contradiction. Thus $n\geq 3.$ By Table \ref{Ta3}, $S$ possesses a unipotent
character $\psi$ with $\psi(1)_p=r^{(n-1)(n-2)/2}.$ It follows that $(n-1)(n-2)b/2<13(2m+1)/2.$
Multiplying both sides by $2n,$ we obtain $bn(n-1)(n-2)=24(2m+1)(n-2)<13n(2m+1).$ After
simplifying, we have $24(n-2)<13n$ and so $n\leq 4.$ If $n=3$ then $b=4(2m+1)$ so that $r=q^8.$
Hence $S=L_3^\epsilon(q^8).$ By \cite{Simp}, $S$ possesses a character of degree $q^{24}-\epsilon.$
However we see that ${}^2F_4(q^2)$ has no such degrees. Thus $n=4$ and so $S=L^\epsilon_4(r),$
where $r=q^4.$ By \cite[$13.8$]{car85}, $S$ possesses a unipotent character of degree
$r^2(r^{n}-\epsilon^n)(r^{n-3}-\epsilon^{n-3})/(r-\epsilon) (r^2-1)=r^2(r^2+1)=q^8(q^8+1).$ However
we can check that $q^8(q^8+1)\not\in \textrm{cd}({}^2F_4(q^2)).$

$(b)$ $S=S_{2n}(r),n\geq 2,S\neq S_4(2)$ or $S=O_{2n+1}(r),n\geq 3,r$ odd. As $p=2,$ $S=S_{2n}(r).$
Then $bn^2=12(2m+1).$ By Table \ref{Ta3}, $S$ possesses a unipotent character $\psi$ with
$\psi(1)_2=2^{b(n-1)^2-1}.$ It follows that $b(n-1)^2-1<13(2m+1)/2$ and so $b(n-1)^2\leq
13(2m+1)/2=13bn^2/24.$ Hence $24(n-1)^2\leq 13n^2.$ Solving this inequality, we obtain $2\leq n\leq
3.$ If $n=2$ then $r=q^6$ and $S$ possesses a character of degree $(r-1)(r^2+1)=(q^6-1)(q^{12}+1)$
by \cite{Eno}. However we can check that $H$ has no such degree. Assume that $n=3.$ Then
$3b=4(2m+1)$ and hence $6b=8(2m+1)>6$ so that $l_{8(2m+1)}(2)$ exists and since $l_{8(2m+1)}(2)\in
\pi(S),$ by Lemma \ref{Lem4}$(iv),$ we deduce that $l_{8(2m+1)}(2)\in \pi(H),$ which is impossible.

$(c)$ $S=O_{2n}^\epsilon(r),n\geq 4.$ We have $bn(n-1)=12(2m+1).$ By \cite[$13.8$]{car85}, $S$
possesses a unipotent character of degree $r(r^n-\epsilon)(r^{n-2}+\epsilon)/(r^2-1).$ If $r=2^m$
then $b=m.$ We have $mn(n-1)=12(2m+1)>24m$ so that $n(n-1)>24$ and hence $n\geq 6.$ Thus
$12(2m+1)=mn(n-1)\geq 30m.$ It follows that $1\leq m\leq 2.$ If $m=1$ then $n(n-1)=36.$ However
this equation has no integer solutions. Thus $m=2.$ Then $r=2^2$ and $n=6.$ Hence
$S=O_{12}^\epsilon(4)$ and $H={}^2F_4(2^5).$ But then $\pi(O_{12}^\epsilon(4))\not\subseteq
\pi({}^2F_4(32)).$ Therefore $2m+1\leq b.$ Thus $n(n-1)\leq 12$ and so as $n\geq 4,$ we deduce that
$n=4.$ Then $b=2m+1$ and hence $r=q^2>2.$ We have $S=O_8^\epsilon(2^{2m+1})$ and
$H={}^2F_4(2^{2m+1}).$ However $l_{3}(2^{2m+1})\in \pi(S)-\pi(H),$ which contradicts Lemma
\ref{Lem4}$(iv).$

$(d)$ $S={}^2B_2(r^2),r^2=2^{2n+1}.$ Then $2(2n+1)=12(2m+1)$ and so $r^2=q^{12}.$ By \cite{Suz},
$S$ possesses a character of degree $r^4+1=q^{24}+1.$ However $q^{24}+1\not\in \textrm{cd}(H).$

$(e)$ $S={}^2F_4(r^2),r^2=2^{2n+1}.$ Then $12(2n+1)=12(2m+1)$ hence
$m=n$ so that $S\cong H.$

$(f)$ $S={}^2G_2(3^{2n+1}).$ This case cannot happen as $S$ is of
characteristic $3.$

$(g)$ $S={}^3D_4(r).$ Then $12b=12(2m+1)$ and so $r=q^{2}.$ By Table \ref{Ta3}, $S$ possesses a
unipotent character $\psi$ with $\psi(1)_p=p^{7b}.$ As $b=2m+1,$ we see that $7b=14(2m+1)>13m+6,$ a
contradiction.

$(h)$ $S={}^2E_6(r)$ or $E_6(r).$ Then $36b=12(2m+1)$ and so $r^3=q^{2}.$ By Table \ref{Ta3}, $S$
possesses a unipotent character $\psi$ with $\psi(1)_p=r^{25}.$ As $r^3=q^2,$ we have
$r^{25}>r^{24}=(r^3)^8=q^{16}>q^{13},$ a contradiction.

$(i)$ $S=G_2(r).$ Then $6b=12(2m+1)$ and so $r=q^{4}.$ As $r$ is even, using \cite{Eno86}, $S$
possesses an irreducible character $\psi\in \textrm{Irr}(S)$ with $\psi(1)_2=r^{3}=q^{12}.$ However
this cannot happen by checking the list of character degrees of $H.$

$(j)$ $S=F_4(r).$ Then $24b=12(2m+1)$ and so $2b=2m+1,$ which is
impossible.


$(k)$ $S=E_7(r).$ Then $63b=12(2m+1)$ and so $21b=4(2m+1).$ By \cite[$13.9$]{car85}, $S$ possesses
a unipotent character of degree $\phi_{7,1}(1)=r\Phi_7\Phi_{12}\Phi_{14}.$ If $r=2^m$ then $b=m$ so
that $21m=4(2m+1)$ and hence $13m=4,$ a contradiction. Thus $2m+1\leq b.$ But then $21b=4(2m+1)\leq
4b,$ which is impossible.

$(l)$ $S=E_8(r).$ Then $120b=12(2m+1)$ and so $10b=(2m+1)$ or $r^{10}=q^2.$ By Table \ref{Ta3}, $S$
possesses a unipotent character $\psi$ with $\psi(1)_p=r^{91}.$ As $r^{10}=q^2,$ we have
$r^{91}>r^{90}=q^{18}>q^{13},$ a contradiction.

{\bf Case $3.$} $H={}^3D_4(q),q=p^a.$ Then $|H|_p=p^{12a}.$ The character degrees of $H$ are
available in \cite[Table $4.4$]{Der}. Observe that if $\chi\in \textrm{Irr}(H)$ with $\chi\neq
St_H$ and $p\mid \chi(1)$ then $q=\chi(1)_p$ or $q^3/(2,q)\mid \chi(1)_p.$ In any cases $q\mid
\chi(1)_p$ and $\chi(1)_p\leq q^7.$

$(a)$ $S=L_n^\epsilon(r),n\geq 2.$ Then $bn(n-1)=24a.$ If $n=2,$ then $b=12a$ hence $r=q^{12}.$ We
have $r+1=q^{12}+1\in \textrm{cd}(S)$ but $q^{12}+1\not\in \textrm{cd}(H),$ a contradiction. Thus
$n\geq 3.$ By Table \ref{Ta3}, $S$ possesses a unipotent character $\psi$ with
$\psi(1)_p=p^{b(n-1)(n-2)/2}.$ Thus $b(n-1)(n-2)/2\leq 7a.$ Multiplying both sides by $n,$ we
obtain $bn(n-1)(n-2)/2=12a(n-2)\leq 7an$ and so $5n\leq 24$ hence $n\leq 4.$ We conclude that
$3\leq n\leq 4.$ If $n=4$ then $r=q^2.$ By \cite[$13.8$]{car85}, $L_4^\epsilon(q^2)$ possesses a
unipotent character labeled by the symbol $(2,2)$ of degree $r^2(r^2+1)=q^4(q^4+1).$ However we can
check that $q^4(q^4+1)\not\in \textrm{cd}(H)$ since $q^4+1\nmid |H|,$ and so
$\textrm{cd}(S)\not\subseteq \textrm{cd}(H).$ If $n=3$ then $r=q^4$ and $S=L^\epsilon_3(q^4).$ By
\cite{Simp}, $S$ possesses a unipotent character of degree $r^3-\epsilon=q^{12}-\epsilon.$ However
we can see that $q^{12}-\epsilon\nmid |H|.$ Thus $\textrm{cd}(S)\not\subseteq \textrm{cd}(H).$

$(b)$ $S=S_{2n}(r),n\geq 2,S\neq S_4(2)$ or $S=O_{2n+1}(r),n\geq 3,r$ odd. We have $bn^2=12a.$ By
\cite[$13.8$]{car85}, $S$ possesses a unipotent character $\chi\in \textrm{Irr}(S)$ labeled by the
symbol $\binom{0\:1\:n}{\:-\:}$ with degree $\chi(1)=r(r^n-1)(r^{n-1}-1)/2(r+1).$

Assume at first that $p$ is odd. If $r=q$ then $b= a.$ It follows that $n^2=12,$ a contradiction.
Thus $q^3\mid \chi(1)_p=r$ and so $b\geq 3a.$ Hence $12a=bn^2\geq 3an^2$ so that $n^2\leq 4$ hence
$n=2.$ Thus $r=q^3$ and so $S=S_4(q^3)$ and $H={}^3D_4(q).$ However using \cite{Der, Sha}, we can
check that $\textrm{cd}(S_4(q^3))\nsubseteq \textrm{cd}({}^3D_4(q)).$

Now assume $p=2.$ Then $S=S_{2n}(r).$ If $r=2$ then $b=1$ so that $n^2=12a.$ By Table \ref{Ta3},
$S$ has a unipotent character $\psi$ with $\psi(1)_2=2^{(n-1)^2-1}.$ Thus $(n-1)^2-1\leq 7a.$
Combining this with $n^2=12a,$ we deduce that $12((n-1)^2-1)\leq 7n^2$ so that $2\leq n\leq 4.$
However for these values of $n,$ the equation $n^2=12a$ is impossible. Thus $r>2$ and so
$\chi(1)_2=2^{b-1}>1.$ If $b-1=a$ then $(a+1)n^2=12a.$ Hence $n^2<12$ so that $n=2$ or $n=3.$ If
$n=2$ then $4(a+1)=12a$ or $2a=1,$ which is impossible. Hence $n=3$ and $9(a+1)=12a,$ which implies
that $a=3.$ Thus $b=4$ and so $S=S_6(16)$ and $H={}^3D_4(8).$ But then $\pi(S)\nsubseteq \pi(H),$
which contradicts Lemma \ref{Lem4}$(iv).$ Thus $q^3/2=2^{3a-1}\mid 2^{b-1}.$ It follows that $b\geq
3a,$ hence $n^2\leq 4$ so that $n=2,$ $b=3a$ and so $r=q^3.$ This leads to a contradiction as in
the case when $p$ is odd.

$(c)$ $S=O_{2n}^\epsilon(r),n\geq 4.$ We have $bn(n-1)=12a.$ By \cite[$13.8$]{car85}, $S$ possesses
a unipotent character $\chi$ of degree $r(r^n-\epsilon)(r^{n-2}+\epsilon)/(r^2-1).$ If $r=q$ then
$a=b$ so that $n(n-1)=12$ or $n=4.$ We have $S=O_8^\epsilon(q)$ and $H={}^3D_4(q).$ We can check
that $\chi(1)\not\in \textrm{cd}(H).$ If $q^3/(2,q)\mid r$ then $b\geq 3a-1\geq 2a$ and so
$12a=bn(n-1)\geq 2an(n-1)$ so that $n(n-1)\leq 6,$ which is impossible as $n\geq 4.$

$(d)$ $S={}^2B_2(r^2),r^2=2^{2n+1}.$ Then $2(2n+1)=12a$ and so
$2n+1=6a,$ which is absurd.

$(e)$ $S={}^2F_4(r^2),r^2=2^{2n+1}.$ Then $12(2n+1)=12a$ and so $2n+1=a.$ By \cite[$13.9$]{car85},
$S$ possesses a unipotent character of degree ${}^2B_2[a](1)=r\Phi_1\Phi_2\Phi_4^2\Phi_6/\sqrt{2}.$
However this is impossible as $r/\sqrt{2}=2^n<q=r^2.$

$(f)$ $S={}^2G_2(r^2),r^2=3^{2n+1}.$ Then $3(2n+1)=12a$ which is
impossible.

$(g)$ $S={}^3D_4(r).$ Then $12b=12a$ and so $r=q.$ Hence $S\cong H.$

$(h)$ $S={}^2E_6(r)$ or $E_6(r).$ Then $36b=12a$ and so $r^3=q.$ By Table \ref{Ta3}, $S$ possesses
a unipotent character $\psi$ with $\psi(1)_p=r^{25}.$ As $r^3=q,$ we have
$r^{25}>r^{21}=(r^3)^7=q^7,$ a contradiction.

$(i)$ $S=G_2(r).$ Then $6b=12a$ and so $r=q^{2}.$ By \cite{Chang,Eno76, Eno86}, $S$ possesses an
irreducible character $\psi$ with $\psi(1)_2=r^{3}=q^{6}.$ However this cannot happen by checking
the list of character degrees of $H.$

$(j)$ $S=F_4(r).$ Then $24b=12a$ and so $q=r^2.$ By \cite[$13.9$]{car85}, $S$ possesses a unipotent
character of degree $\phi_{9,10}(1)=r^{10}\Phi_3^2\Phi_6^2\Phi_{12}.$ As $r^{10}=q^5,$ we can check
that this degree does not belong to $\textrm{cd}(H).$

$(k)$ $S=E_7(r).$ Then $63b=12a$ and so $21b=4a.$ By \cite[$13.9$]{car85}, $S$ possesses a
unipotent character of degree $\phi_{7,1}(1)=r\Phi_7\Phi_{12}\Phi_{14}.$ It follows that $b\geq a$
and so $4a=21b\geq 21a,$ which is impossible.

$(l)$ $S=E_8(r).$ Then $120b=12a$ and so $10b=a$ or $r^{10}=q.$ By Table \ref{Ta3}, $S$ possesses a
unipotent character $\psi$ with $\psi(1)_p=r^{91}.$ As $r^{10}=q,$ we have
$r^{91}>r^{90}=q^{9}>q^{7},$ a contradiction.

{\bf Case $4.$} $H={}^2E_6(q),q=p^a.$ Then $|H|_p=p^{36a}.$ If $\chi\in \textrm{Irr}(H)$ with
$\chi\neq St_H,$ then $\chi(1)_p\leq p^{25a}.$

$(a)$ $S=L_n^\epsilon(r),r=p^b, n\geq 2.$ Then $bn(n-1)=72a.$ If $n=2,$ then $b=36a$ hence
$r=q^{36}.$ We have $r+1=q^{36}+1\in \textrm{cd}(S)$ but $q^{36}+1\nmid |H|,$ a contradiction. Thus
$n\geq 3.$ By Table \ref{Ta3}, $S$ has a unipotent character $\psi\in \textrm{Irr}(S)$ with
$\psi(1)_p=p^{b(n-1)(n-2)/2}.$ It follows that $b(n-1)(n-2)\leq 50a.$ Multiplying both sides by
$n,$ we obtain $bn(n-1)(n-2)=72a(n-2)\leq 50an$ hence $11n\leq 72$ so that $n\leq 6.$ Assume first
that $l_{\epsilon bn}(p)$ exists. As $r^n-\epsilon^n\mid |S|,$ we deduce that $l_{\epsilon
bn}(p)\in \pi(H)$ by Lemma \ref{Lem4}$(iv).$ Thus $bn\leq 18a.$ Multiplying both sides by $n-1,$ we
obtain $bn(n-1)=72a\leq 18a(n-1),$ hence $n\geq 5.$ Thus $5\leq n\leq 6.$ If $n=5$ then $5b=18a.$
It follows that $4b=72a/5>12a$ so that $l_{4b}(p)\in \pi(S)\subseteq\pi(H)$ exists and hence
$4b\mid 18a=5b,$ which is impossible. If $n=6$ then $5b=12a.$ We also have $4b=48a/5>9a$ and so
arguing as above, we have $l_{4b}\in \pi(H).$ It follows that $4b\mid 18a$ or $4b\mid 12a=5b.$
Obviously the latter case cannot happen and so $4b\mid 18a$ and then $8b\mid 36a=15b,$ which is
impossible. Thus $l_{\epsilon bn}(p)$ does not exist. It follows that $bn=6$ or $2bn=6.$ In both
cases, we have $bn\leq 6$ so that $n-1=72a/bn\geq 72a/6=12a\geq 12,$ which is a contradiction as
$n\leq 6.$

$(b)$ $S=S_{2n}(r),n\geq 2,S\neq S_4(2)$ or $S=O_{2n+1}(r),n\geq 3,r$ odd. We have $bn^2=36a.$ If
$p=2$ and $2bn=6,$ then $bn=3$ so that $n=3,b=1$ and hence $9=36a,$ which is absurd. Thus
$l_{2bn}(p)$ exists and so $l_{2bn}(p)\in\pi(H)$ as $l_{2bn}(p)\in \pi(S).$ We deduce that $2bn\leq
18a$ or $bn\leq 9a.$ Multiplying both sides by $n,$ we obtain $bn^2=36a\leq 9an.$ Hence $n\geq 4.$
By Table \ref{Ta3}, $S$ has a unipotent character $\psi$ with $\psi(1)_p\geq p^{b(n-1)^2-b}$ and so
$bn(n-2)\leq 25a.$ Multiplying both sides by $n,$ we have $bn^2(n-2)=36a(n-2)\leq 25an.$ Thus
$4\leq n\leq 6.$ If $n=4$ then $4b=9a.$ We have $6b=27a/2>13a$ and so
$l_{6b}(p)\in\pi(S)\subseteq\pi(H)$ exists and hence $6b\mid 18a=8b,$ which is impossible.
Similarly, if $n=5$ then $25b=36a.$ We have $10b=360a/25=72a/5>12a$ and so $l_{10b}(p)\in \pi(H).$
Hence $10b\mid 18a,$ and so $20b\mid 36a=25b,$ a contradiction. If $n=6$ then $b=a$ so that $q=r.$
However we see that $l_5(q)\in \pi(S)$ but $l_5(q)\not\in \pi(H),$ which contradicts Lemma
\ref{Lem4}$(iv).$

$(c)$ $S=O_{2n}^\epsilon(r),n\geq 4.$ We have $bn(n-1)=36a.$ By Table \ref{Ta3}, $S$ has a
unipotent character $\psi$ with $\psi(1)_p\geq p^{b(n-1)(n-2)}.$ Thus $b(n-1)(n-2)\leq 25a.$
Multiplying both sides by $n$ and simplifying, we obtain $11n\leq 72$ and so $4\leq n\leq 6.$ If
$n=4$ then $b=3a$ and hence $r=q^3.$ We have $S=O^\epsilon_8(q^3)$ and $H={}^2E_6(q).$ We have
$l_9(q)\in \pi(S)$ but $l_9(q)\not\in\pi(H),$ which contradicts Lemma \ref{Lem4}$(iv).$ If $n=5$
then $5b=9a.$ As $8b=72a/5>12a,$ $l_{8b}(p)\in\pi(S)\subseteq\pi(H)$ exists and hence $8b\mid
18a=10b,$ which is impossible. Assume that $n=6.$ Then $5b=6a.$ As above, we have $l_{8b}(p)\in
\pi(H)$ and since $8b=48a/5>9a,$ we have $8b\mid 10a,$ $8b\mid 12a=10b$ or $8b\mid 18a=15b.$
Obviously the last two cases cannot happen. If $8b\mid 10a,$ then $24b\mid 30a=25b,$ which is again
impossible.

$(d)$ $S={}^2B_2(r^2),r^2=2^{2n+1}$ or ${}^2G_2(r^2),r^2=3^{2n+1}.$
Then equation $|S|_p=|H|_p$ cannot happen.

$(e)$ $S={}^2F_4(r^2),r^2=2^{2n+1}.$ Then $12(2n+1)=36a$ and so
$2n+1=3a$ or $r^2=q^3.$ We have $r^{12}+1=q^{18}+1\mid |S|$ so that
$l_{36}(q)\in \pi(H),$ which is impossible.

$(f)$ $S={}^3D_4(r).$ Then $12b=36a$ and so $r=q^3.$ We observe that
$q^{36}-1=r^{12}-1=(r^4-1)(r^8+r^4+1)$ and $r^8+r^4+1\mid |S|$ so
that $l_{36}(q)\mid |S|$ but $l_{36}(q)\not\in \pi(H),$ a
contradiction.

$(g)$ $S={}^2E_6(r).$ Then $36b=36a$ and so $r=q.$ Thus $S\cong H.$

$(h)$ $S=G_2(r).$ Then $6b=36a$ and so $r=q^{6}.$ We have
$r^6-1=q^{36}-1\mid |S|$ so that $l_{36}(q)\in \pi(H),$ which is
impossible.

$(i)$ $S=F_4(r).$ Then $24b=36a$ and so $q^3=r^2.$ We have
$r^{6}-1=q^{9}-1\mid |S|$ so that $l_9(q)\in \pi(H).$ However
$l_9(q)\nmid |H|,$ a contradiction.

$(j)$ $S=E_6(r).$ Then $36b=36a$ and so $r=q.$ We have $l_9(q)\in
\pi(S)$ but $l_9(q)\not\in \pi(H),$ a contradiction.

$(k)$ $S=E_7(r).$ Then $63b=36a$ and so $7b=4a.$ By Table \ref{Ta3},
$S$ has a unipotent character $\psi$ with $\psi(1)_p=r^{46}.$ As
$r^7=q^4,$ we have $r^{46}=q^{184/7}>q^{25},$ a contradiction.

$(l)$ $S=E_8(r).$ Then $120b=36a$ and so $20b=9a$ or $r^{20}=q^9.$
By Table \ref{Ta3}, $S$ has a unipotent character $\psi$ with
$\psi(1)_p=r^{91}.$ As $r^{20}=q^9,$ we have
$r^{91}=q^{819/20}>q^{25},$ a contradiction.

{\bf Case $5.$} $H=G_2(q),q=p^a.$ Then $|H|_p=p^{6a}.$ The character degrees of $G_2(q)$ are
available in \cite{Chang,Eno76, Eno86}. Observe that if $\chi\in \textrm{Irr}(H)$ with $\chi\neq
St_H$ and $p\mid \chi(1)$ then $q/(6,q)\mid \chi(1)_p$ and $\chi(1)_p\leq p^{3a}.$

$(a)$ $S=L_n^\epsilon(r),n\geq 2.$ Then $bn(n-1)=12a.$ If $n=2,$ then $b=6a$ hence $r=q^{6}.$ We
have $r+1=q^{6}+1\in \textrm{cd}(S)$ but $q^{6}+1\nmid |H|,$ a contradiction. Thus $n\geq 3.$ By
Table \ref{Ta3}, $S$ possesses a unipotent character $\psi\in \textrm{Irr}(S)$ with
$\psi(1)_p=p^{b(n-1)(n-2)/2}.$ It follows that $b(n-1)(n-2)\leq 6a.$ Multiplying both sides by $n$
and simplifying, we obtain $3\leq n\leq 4.$ If $n=4$ then $b=a$ so that $S=L_4^\epsilon(q)$ and
$H=G_2(q).$ However $l_4(q)\in \pi(S)$ but $l_4(q)\not\in\pi(H).$ If $n=3$ then $b=2a$ so that
$r=q^2.$ We have $S=L_3^\epsilon(q^2)$ and $S$ possesses a character of degree $q^2(q^2+\epsilon).$
However, we can check that this degree does not belong to $\textrm{cd}(H).$

$(b)$ $S=S_{2n}(r),n\geq 2,S\neq S_4(2)$ or $S=O_{2n+1}(r),n\geq 3,r$ odd. We have $bn^2=6a.$ By
Table \ref{Ta3}, $S$ possesses a unipotent character $\psi\in \textrm{Irr}(S)$ with $\psi(1)_p\geq
p^{bn(n-2)}.$ Thus $bn(n-2)\leq 3a.$ It follows that $2\leq n\leq 4.$ If $n=2$ then $r^2=q^3$ and
$S=S_4(r)$ possesses a unipotent character labeled by the symbol $\binom{1\:2}{\:0\:}$ with degree
$r(r^2+1)/2.$ Assume first that $p$ is odd. Since $r=q^{3/2}>q,$ we deduce from the list of
character degrees of $G_2(q)$ that $r=q^2$ or $r=q^3.$ But then $r\geq q^{2}>q^{3/2},$ a
contradiction. Thus $p=2.$ We have $r/2=\frac{1}{2}q^{3/2}\geq q.$ Thus $r/2\in \{q,q^2,q^3\}.$ If
$r=2q,$ then $q=4$ and so $S=S_4(8)$ and $H=G_2(4).$ However by using \cite{GAP},
$\textrm{cd}(S_4(8))\nsubseteq \textrm{cd}(G_2(4)).$ For the remaining cases, we have $r\geq 2q^2$
so that $r^2=q^3\geq 4q^4,$ which is impossible. If $n=3$ then $3b=2a.$ We have $l_{4b}(p)\in
\pi(S)$ so that $l_{4b}(p)\in \pi(H).$ As $4b>2a,$ we deduce that  $4b\mid 6a=9b,$ which is
impossible. If $n=4$ then $8b=3a.$ If $l_{6b}(p)$ exists then $l_{6b}(b)\in \pi(H).$ As
$6b=18a/8>2a,$ we have $6b\mid 6a=16b,$ which is impossible. Hence $6b=6$ and $p=2.$ Hence $b=1$
and $3a=8,$ a contradiction.

$(c)$ $S=O_{2n}^\epsilon(r),n\geq 4.$ We have $bn(n-1)=6a.$ By Table \ref{Ta3}, $S$ has a unipotent
character $\psi$ with $\psi(1)_p\geq p^{b(n-1)(n-2)}.$ Thus $b(n-1)(n-2)\leq 3a.$ Multiplying both
sides by $n$ and simplifying, we obtain $n\leq 4$ and so $n=4$ and $2b=a$ and hence $q=r^2.$ We
have $S=O^\epsilon_8(r)$ and $H=G_2(r^2).$ By \cite[$13.8$]{car85}, $S$ possesses a unipotent
character $\varphi$ of degree $r(r^4-\epsilon)(r^{2}+\epsilon)/(r^2-1).$ Thus $a-1\leq b$ so that
$a=2b\geq 2(a-1)$ hence $a\leq 2.$ Since $a=2b\geq 2,$ we obtain $a=2$ and $b=1.$ Then
$S=O_8^\epsilon(2)$ and $H=G_2(4).$ Using \cite{atlas} we can see that
$\textrm{cd}(O_8^\epsilon(2))\not\subseteq \textrm{cd}(G_2(4)).$

$(d)$ $S={}^2B_2(r^2),r^2=2^{2n+1}.$ Then $2(2n+1)=6a,$ hence $2n+1=3a.$ By \cite{Suz}, $S$
possesses a character of degree $r^4+1=q^6+1.$ But then $q^6+1\nmid |H|,$ a contradiction.

$(e)$ $S={}^2F_4(r^2),r^2=2^{2n+1}.$ Then $12(2n+1)=6a$ and so
$2(2n+1)=a.$ By Table \ref{Ta3}, $S$ has a unipotent character
$\psi$ with $\psi(1)_2=r^{13n+6}.$ Hence $13n+6\leq 3a=6(2n+1).$ It
follows that $13n+6\leq 12n+6,$  which is impossible.

$(f)$ $S={}^2G_2(r^2),r^2=3^{2n+1}.$ Then $3(2n+1)=6a,$ which is impossible.

$(g)$ $S={}^3D_4(r).$ Then $12b=6a$ and so $a=2b.$ By \cite[$13.9$]{car85}, $S$ possesses a
unipotent character of degree $\phi_{1,3}'(1)=r\Phi_{12}.$ Thus $b\geq a-1$ and hence $a=2b\geq
2(a-1)$ so that $a\leq 2.$ As $a=2b\geq 2,$ we have $a=2$ and $b=1.$ Then $S={}^3D_4(2)$ and
$H=G_2(4).$ Using \cite{atlas} we can see that $\textrm{cd}({}^3D_4(2))\not\subseteq
\textrm{cd}(G_2(4)).$

$(h)$ $S={}^2E_6(r).$ Then $36b=6a$ and so $a=6b.$ By \cite[$13.9$]{car85}, $S$ possesses a
unipotent character of degree $\phi_{2,4}'(1)=r\Phi_8\Phi_{12}.$ Thus $b\geq a-1$ and hence
$a=6b\geq 6(a-1)$ so that $a=1,$ which is impossible as  $a=6b\geq 6.$

$(i)$ $S=G_2(r).$ Then $6b=6a$ and so $r=q.$ Thus $S\cong H.$

$(j)$ $S=F_4(r).$ Then $24b=6a$ and so $a=4b.$ By \cite[$13.9$]{car85}, $S$ possesses a unipotent
character of degree $\phi_{9,2}(1)=r^2\Phi_{3}^2\Phi_6^2\Phi_{12}.$ Thus $2b\geq a-1$ and hence
$a=4b\geq 2(a-1)$ so that $a\leq 2,$ which is impossible as $a=4b\geq 4.$

$(k)$ $S=E_6(r).$ Then $36b=6a$ and so $a=6b.$ By \cite[$13.9$]{car85}, $S$ possesses a unipotent
character of degree $\phi_{6,1}(1)=r\Phi_{8}\Phi_{9}.$ Thus $b\geq a-1$ and hence $a=6b\geq 6(a-1)$
so that $a\leq 1,$ which is impossible as $a=6b\geq 6.$

$(l)$ $S=E_7(r).$ Then $63b=6a$ and so $2a=21b.$ By \cite[$13.9$]{car85}, $S$ possesses a unipotent
character of degree $\phi_{7,1}(1)=r\Phi_{7}\Phi_{12}\Phi(14).$ Thus $b\geq a-1$ and hence
$2a=21b\geq 21(a-1)$ so that $a\leq 1,$ which is impossible as $a=21b/2>10.$

$(m)$ $S=E_8(r).$ Then $120b=6a$ and so $a=20b.$ By \cite[$13.9$]{car85}, $S$ possesses a unipotent
character of degree $\phi_{8,1}(1)=r\Phi_4^2\Phi_{8}\Phi_{12}\Phi_{20}\Phi_{24}.$ Thus $b\geq a-1$
and hence $a=20b\geq 20(a-1)$ so that $a\leq 1,$ which is impossible as $a=20b\geq 20.$

{\bf Case $6.$} $H=F_4(q),q=p^a.$ Then $|H|_p=p^{24a}.$ If $\chi\in \textrm{Irr}(H)$ with $\chi\neq
St_H$ then $\chi(1)_p\leq p^{16a}.$

$(a)$ $S=L_n^\epsilon(r),n\geq 2.$ Then $bn(n-1)=48a.$ If $n=2,$ then $b=24a$ hence $r=q^{24}.$ We
have $r+1=q^{24}+1\in \textrm{cd}(S)$ but $q^{24}+1\nmid |H|,$ a contradiction. Thus $n\geq 3.$ By
Table \ref{Ta3}, $S$ possesses a unipotent character $\psi$ with $\psi(1)_p=p^{b(n-1)(n-2)/2}.$ It
follows that $b(n-1)(n-2)\leq 32a.$ Multiplying both sides by $n$ and simplifying, we obtain $3\leq
n\leq 6.$ Assume that $l_{\epsilon bn}(p)$ exists. Then $l_{\epsilon bn}(p)\in \pi(H)$ by Lemma
\ref{Lem4}$(iv).$ Thus $bn\leq 12a.$ Multiplying both sides by $n-1,$ we have $bn(n-1)=48a\leq
12a(n-1),$ and hence $n\geq 5.$ Thus $5\leq n\leq 6.$ If $n=5$ then $5b=12a.$ It follows that
$4b=48a/5>9a$ so that $l_{4b}(p)$ exists and since $l_{4b}(p)\in \pi(S),$ by Lemma
\ref{Lem4}$(iv),$ $l_{4b}(p)\in \pi(H)$ so that $4b\mid 12a=5b,$ which is impossible. If $n=6$ then
$5b=8a.$ As $6b=48a/5>9a,$ $l_{6b}$ exists and so as above $l_{6b}(p)\in\pi(H)$  It follows that
$6b\mid 12a$ and hence $12b\mid 24a=15b,$  which is impossible. Thus $l_{\epsilon bn}(p)$ does not
exist and $3\leq n\leq 6.$ It follows that  $bn=6$ or $2bn=6.$ In both cases we obtain $bn\leq 6$
and so $n-1=48a/bn\geq 48a/6=8a\geq 8,$ which is a contradiction as $n\leq 6.$

$(b)$ $S=S_{2n}(r),n\geq 2,S\neq S_4(2)$ or $S=O_{2n+1}(r),n\geq 3,r$ odd. We have $bn^2=24a.$ If
$2bn=6,$ then $bn=3$ so that $n=3,b=1$ and hence $9=24a,$ which is absurd. Thus
$l_{2bn}(p)\in\pi(S)\subseteq\pi(H)$ exists and so $2bn\leq 12a$ or equivalently $bn\leq 6a.$
Multiplying both sides by $n,$ we obtain $bn^2=24a\leq 6an.$ Hence $n\geq 4.$ By Table \ref{Ta3},
$S$ has a unipotent character $\psi$ with $\psi(1)_p\geq p^{bn(n-2)}$ and so $bn(n-2)\leq 16a.$
Multiplying both sides by $n,$ we have $bn^2(n-2)=24a(n-2)\leq 16an.$ Thus $4\leq n\leq 6.$ If
$n=4$ then $2b=3a.$ We have $6b=9a$ and so $l_{6b}(p)$ exists. Hence $6b\mid 12a=8b,$ which is
impossible. If $n=5$ then $25b=24a.$ We have $10b=240a/25=48a/5>9a$ and so arguing as above, we
have $10b\mid 12a.$ Hence $40b\mid 48a=50b,$ a contradiction. If $n=6$ then $3b=2a.$ We have
$10b=20a/3>6a$ so that $l_{10b}(a)$ exists and so $10b\mid 8a=12b$ or $10b\mid 12a=18b.$ However we
can see that both cases are impossible.

$(c)$ $S=O_{2n}^\epsilon(r),n\geq 4.$ We have $bn(n-1)=24a.$ By Table \ref{Ta3}, $S$ has a
unipotent character $\psi$ with $\psi(1)_p\geq p^{b(n-1)(n-2)}.$ Thus $b(n-1)(n-2)\leq 16a.$
Multiplying both sides by $n$ and simplifying, we obtain $3(n-2)\leq 2n$ and so $4\leq n\leq 6.$ If
$n=4$ then $b=2a$ and hence $r=q^2.$ We have $S=O^\epsilon_8(q^2)$ and $H=F_4(q).$ By
\cite[$13.8$]{car85}, $S$ possesses a unipotent character $\varphi$ of degree
$r(r^4-\epsilon)(r^{2}+\epsilon)/(r^2-1).$ However we can check that $\varphi(1)\nmid |H|.$ If
$n=5$ then $5b=6a.$ As $8b=48a/5>9a,$ $l_{8b}(p)\in \pi(S)$ so that $l_{8b}(p)\in \pi(H).$  We
deduce that $8b\mid 12a=10b,$ which is impossible. If $n=6$ then $5b=4a.$ As above, we have
$8b=32a/5>6a,$ so that $l_{8b}(p)\in \pi(H)$ and hence  $8b\mid 8a=10b$ or $8b\mid 12a=15b.$
Obviously both cases are impossible.

$(d)$ $S={}^2B_2(r^2),{}^2F_4(r^2)$ or ${}^2G_2(r^2),$ where $r^2=2^{2n+1},2^{2n+1}$ or $3^{2n+1},$
respectively. In these cases, the equation $|S|_p=|H|_p$ cannot happen.

$(e)$ $S={}^3D_4(r).$ Then $12b=24a$ and so $r=q^2.$ We observe that
$q^{24}-1=r^{12}-1=(r^4-1)(r^8+r^4+1)$ and $r^8+r^4+1\mid |S|$ so
that $l_{24}(q)\mid |S|$ but $l_{24}(q)\not\in \pi(H),$ which contradicts Lemma \ref{Lem4}$(iv).$

$(f)$ $S={}^2E_6(r)$ or $E_6(r).$ Then $36b=24a$ and so $3b=2a.$ By Table
\ref{Ta3}, $S$ has a unipotent character $\psi$ with
$\psi(1)_p=p^{25b}.$ Thus $25b\leq 16a.$ Hence $25b\leq 24b,$ a
contradiction.

$(g)$ $S=G_2(r).$ Then $6b=24a$ and so $r=q^{4}.$ We have
$r^6-1=q^{24}-1\mid |S|$ so that $l_{24}(q)\in \pi(H),$ which is
impossible.

$(h)$ $S=F_4(r).$ Then $24b=24a$ and so $q=r$ so that $S\cong H.$

 $(i)$ $S=E_7(r).$ Then
$63b=24a$ and so $21b=8a.$ By Table \ref{Ta3}, $S$ possesses a unipotent character $\psi$ with
$\psi(1)_p=r^{46}.$ As $r^{21}=q^8,$ we have $r^{46}=q^{8.46/24}>q^{16},$ a contradiction.

$(j)$ $S=E_8(r).$ Then $120b=24a$ and so $a=5b$ or $q=r^5.$  By Table \ref{Ta3}, $S$ possesses a
unipotent character $\psi$ with $\psi(1)_p=r^{91}.$ As $r^{5}=q,$ we have
$r^{91}>r^{90}=q^{18}>q^{16},$ a contradiction.

{\bf Case $7.$} $H=E_6(q),q=p^a.$ Then $|H|_p=p^{36a}.$ If $\chi\in \textrm{Irr}(H)$ with $\chi\neq
St_H,$ then  $\chi(1)_p\leq p^{25a}.$

$(a)$ $S=L_n^\epsilon(r),n\geq 2.$ Then $bn(n-1)=72a.$ If $n=2,$ then $b=36a$ hence $r=q^{36}.$ We
have $r+1=q^{36}+1\in \textrm{cd}(S)$ but $q^{36}+1\nmid |H|,$ a contradiction. Thus $n\geq 3.$ By
Table \ref{Ta3}, $S$ has a unipotent character $\psi\in \textrm{Irr}(S)$ with
$\psi(1)_p=p^{b(n-1)(n-2)/2}.$ It follows that $b(n-1)(n-2)\leq 50a.$ Multiplying both sides by $n$
and simplifying, we obtain $11n\leq 72$ so that $n\leq 6.$ Assume first that $l_{\epsilon bn}(p)$
exists. We deduce that $l_{\epsilon bn}(p)\in \pi(S)\subseteq\pi(H)$ by Lemma \ref{Lem4}$(iv).$
Thus $bn\leq 12a.$ Multiplying both sides by $n-1,$ we have $bn(n-1)=72a\leq 12a(n-1).$ Hence
$n\geq 7,$ a contradiction. Thus $l_{\epsilon bn}(p)$ does not exist. Arguing as in Case $4(a),$ we
get a contradiction.

$(b)$ $S=S_{2n}(r),n\geq 2,S\neq S_4(2)$ or $S=O_{2n+1}(r),n\geq 3,r$ odd. We have $bn^2=36a.$ If
$p=2$ and $2bn=6,$ then $bn=3$ so that $n=3,b=1$ and hence $9=36a,$ which is absurd. Thus
$l_{2bn}(p)\in \pi(S)\subseteq\pi(H)$ exists and so $2bn\leq 12a$ or equivalently $bn\leq 6a.$
Multiplying both sides by $n$ and simplifying, we obtain $n\geq 6.$ By Table \ref{Ta3}, $S$ has a
unipotent character $\psi$ with $\psi(1)_p\geq p^{b(n-1)^2-b}$ and so $bn(n-2)\leq 25a.$
Multiplying both sides by $n$ and simplifying, we obtain $n\leq 6.$ Thus $n=6$ and then $b=a$ so
that $q=r.$ However $l_{10}(q)\in \pi(S)- \pi(H),$ contradicting Lemma \ref{Lem4}$(iv).$

$(c)$ $S=O_{2n}^\epsilon(r),n\geq 4.$ We have $bn(n-1)=36a.$ By Table \ref{Ta3}, $S$ has a
unipotent character $\psi$ with $\psi(1)_p\geq p^{b(n-1)(n-2)}.$ Thus $b(n-1)(n-2)\leq 25a.$
Multiplying both sides by $n$ and simplifying, we obtain $11n\leq 72$ and so $4\leq n\leq 6.$ If
$2(n-1)b=6$ then $b=1$ and $n=4$ so that $36a=12,$ which is impossible. Thus
$l_{2(n-1)b}(p)\in\pi(S)\subseteq\pi(H)$ exists and hence $2b(n-1)\leq 12a.$ Multiplying both sides
by $n$ and simplifying, we obtain $n\geq 6.$ Therefore $n=6$ and so $5b=6a.$ As $8b=48a/5>9a,$
$l_{8b}(p)$ exists and lies in $\pi(S)$ so that by Lemma \ref{Lem4}$(iv),$ $l_{8b}(p)\in \pi(H).$
Since $8b>9a,$ we must have  $8b\mid 12a=10b,$ which is impossible.

$(d)$ $S={}^2E_6(r).$ Then $36b=36a$ and so $r=q.$ Thus
$S={}^2E_6(q)$ and $H=E_6(q).$ However we have $l_{18}(q)\in
\pi(S)$ but $l_{18}(q)\not\in \pi(H),$ which contradicts Lemma  \ref{Lem4}$(iv).$

$(e)$ $S=E_6(r).$ Then $36b=36a$ and so $r=q.$ Hence $S\cong H.$

For the remaining simple exceptional groups of Lie type, using the same argument as in Case $4,$ we
deduce that $\textrm{cd}(S)\not\subseteq \textrm{cd}(H)$ in any of these cases.

{\bf Case $8.$} $H=E_7(q),q=p^a.$ Then $|H|_p=p^{63a}.$ If $\chi\in \textrm{Irr}(H)$ with $\chi\neq
St_H,$ then $\chi(1)_p\leq p^{46a}.$

$(a)$ $S=L_n^\epsilon(r),r=p^b.$ Then $bn(n-1)=126a.$ If $n=2,$ then $b=63a$ hence $r=q^{63}.$ We
have $r+1=q^{63}+1\in \textrm{cd}(S)$ but $q^{63}+1\nmid |H|,$ a contradiction. Thus $n\geq 3.$ By
Table \ref{Ta3}, $S$ has a unipotent character $\psi\in \textrm{Irr}(S)$ with
$\psi(1)_p=p^{b(n-1)(n-2)/2}.$ It follows that $b(n-1)(n-2)\leq 92a.$ Multiplying both sides by $n$
and simplifying, we obtain $n\leq 7.$ Assume first that $l_{\epsilon bn}(p)$ exists. We deduce that
$l_{\epsilon bn}(p)\in \pi(H)$ by Lemma \ref{Lem4}$(iv).$ Thus $bn\leq 18a.$ Multiplying both sides
by $n-1$ and simplifying, we obtain $n\geq 8,$ a contradiction. Thus $l_{\epsilon bn}(p)$ does not
exist. It follows that  $bn=6$ or $bn=3.$ In both cases, we have $bn\leq 6$ so that
$n-1=126a/bn\geq 126a/6=21a\geq 21,$ a contradiction as $n\leq 7.$

$(b)$ $S=S_{2n}(r),n\geq 2,S\neq S_4(2)$ or $S=O_{2n+1}(r),n\geq 3,r$ odd. We have $bn^2=63a.$ If
$p=2$ and $2bn=6,$ then $bn=3$ so that $n=3,b=1$ and hence $9=63a,$ which is impossible. Thus
$l_{2bn}(p)\in\pi(S)\subseteq \pi(H)$ exists and hence $2bn\leq 18a$ so that $bn\leq 9a.$
Multiplying both sides by $n$ and simplifying, we obtain $n\geq 7.$ By Table \ref{Ta3}, $S$ has a
unipotent character $\psi$ with $\psi(1)_p\geq p^{b(n-1)^2-b}$ and so $bn(n-2)\leq 46a.$
Multiplying both sides by $n$ and simplifying, we obtain $n\leq 7.$ Thus $n=7$ and then $7b=9a.$ We
see that $12b=108a/7>15a$ so that $l_{12b}(p)\in \pi(S)\subseteq\pi(H)$ exists and hence $12b\mid
18a=14b,$ which is impossible.

$(c)$ $S=O_{2n}^\epsilon(r),n\geq 4.$ We have $bn(n-1)=63a.$ By Table \ref{Ta3}, $S$ has a
unipotent character $\psi$ with $\psi(1)_p\geq p^{b(n-1)(n-2)}.$ Thus $b(n-1)(n-2)\leq 46a.$
Multiplying both sides by $n$ and simplifying, we obtain $17n\leq 126$ and so $4\leq n\leq 7.$ If
$2(n-1)b=6$ then $b=1$ and $n=4$ so that $63a=12$ which is impossible. Thus $l_{2(n-1)b}(p)\in
\pi(S)\subseteq \pi(H)$ exists and hence $2b(n-1)\leq 18a.$ Multiplying both sides by $n$ and
simplifying, we obtain $n\geq 7.$ Therefore $n=7$ and so $2b=3a.$ We have $10b=15a$ so that
$l_{10b}(p)\in\pi(S)\subseteq\pi(H)$ exists so that $10b\mid 18a=12b,$ which is impossible.

$(d)$ $S={}^2B_2(r^2).$ Then  $r^4=q^{63}.$ By \cite{Suz}, $S$ possesses a character of degree
$r^4+1=q^{63}+1.$ But then $q^{63}+1\nmid |H|,$ a contradiction.

$(e)$ $S={}^2F_4(r^2),r^2=2^{2n+1}.$ Then $12(2n+1)=63a.$  We have $l_{12(2n+1)}(2)\in \pi(S)$
so that by Lemma \ref{Lem4}, $l_{12(2n+1)}(2)\in \pi(H).$ We have
$12(2n+1)=63a$ and so $l_{63a}(2)\in \pi(H),$ and then $l_{63}(q)\in
\pi(H),$ which is impossible.

$(f)$ $S={}^2G_2(r^2),r^2=3^{2n+1}.$ Then $3(2n+1)=63a$ so that
$2n+1=21a.$  We have $l_{6(2n+1)}(3)\in \pi(S)$ so
that $l_ {6(2n+1)}(3)\in\pi(H).$ As $6(2n+1)=126a,$ we deduce that
$l_{126a}(3)\in \pi(H),$ which is impossible.

$(g)$ $S={}^3D_4(r).$ Then $12b=63a$ and so $4b=21a$ or
$r^4=q^{21}.$ We observe that $q^{36}-1=r^{12}-1=(r^4-1)(r^8+r^4+1)$
and $r^8+r^4+1\mid |S|$ so that $l_{36}(q)\mid |S|$ but
$l_{36}(q)\not\in \pi(H),$ which contradicts Lemma \ref{Lem4}$(iv).$

$(h)$ $S={}^2E_6(r).$ Then $36b=63a$ and so $4b=7a.$ We have
$12b=21a$ and $p^{12b}-1\mid |S|$ so that $l_{21}(q)\in \pi(H),$ which is impossible as $21a>18a.$

$(i)$ $S=G_2(r).$ Then $6b=63a$ and so $2b=21a.$ We have
$r^6-1=q^{63}-1\mid |S|$ so that $l_{63}(q)\in \pi(H),$ which is
impossible.

$(j)$ $S=F_4(r).$ Then $24b=63a$ and so $8b=21a.$ As
$p^{8b}-1=q^{21}-1\mid |S|,$ we deduce that $l_{21}(q)\in \pi(H),$
which is impossible.

$(k)$ $S=E_6(r).$ Then $36b=63a$ and so $4b=7a.$ We have
$p^{12b}-1=q^{21}-1\mid |S|$ so that $l_{21}(q)\in \pi(H),$ a
contradiction.

$(l)$ $S=E_7(r).$ Then $63b=63a$ and so $b=a$ or $q=r.$ Hence
$S\cong H.$

$(m)$ $S=E_8(r).$ Then $120b=63a$ and so $40b=21a.$ As $l_{30b}(p)$
exists and belongs to $\pi(S),$ it follows that $l_{30b}(p)\in
\pi(H).$ As $30b=63a/4>15a,$ we deduce that $30b\mid 18a$ and hence
$120b=63a\mid 72a,$ which is impossible.

{\bf Case $9.$} $H=E_8(q),q=p^a.$ Then $|H|_p=p^{120a}.$ If $\chi\in \textrm{Irr}(H)$ with
$\chi\neq St_H,$ then $\chi(1)_p\leq p^{91a}.$

$(a)$ $S=L_n^\epsilon(r),r=p^b.$ Then $bn(n-1)=240a.$ If $n=2,$ then $b=120a$ hence $r=q^{120}.$ We
have $r+1=q^{120}+1\in \textrm{cd}(S)$ but $q^{120}+1\nmid |H|,$ a contradiction. Thus $n\geq 3.$
By Table \ref{Ta3}, $S$ has a unipotent character $\psi\in \textrm{Irr}(S)$ with
$\psi(1)_p=p^{b(n-1)(n-2)/2}.$ It follows that $b(n-1)(n-2)\leq 182a.$ Multiplying both sides by
$n$ and simplifying, we obtain $n\leq 8.$ Assume first that $l_{\epsilon bn}(p)$ exists. We deduce
that $l_{\epsilon bn}(p)\in \pi(H)$ by Lemma \ref{Lem4}$(iv).$ Thus $bn\leq 30a.$ Multiplying both
sides by $n-1,$ we have $bn(n-1)=240a\leq 30a(n-1).$ Hence $n\geq 9,$ which contradicts the
assertion $n\leq 8$ above. Thus $l_{\epsilon bn}(p)$ does not exist. It follows that  $bn=6$ or
$bn=3,$ and hence $bn\leq 6$ so that $n-1=240a/bn\geq 240a/6=40a\geq 40,$ a contradiction.

$(b)$ $S=S_{2n}(r),n\geq 2,S\neq S_4(2)$ or $S=O_{2n+1}(r),n\geq 3,r$ odd. We have $bn^2=120a.$ If
$p=2$ and $2bn=6,$ then $bn=3$ so that $n=3,b=1$ and hence $9=120a,$ which is impossible. Thus
$l_{2bn}(p)\in\pi(S)\subseteq\pi(H)$ exists and hence $2bn\leq 30a$ so that $bn\leq 15a.$
Multiplying both sides by $n,$ we obtain $bn^2=120a\leq 15an.$ Hence $n\geq 8.$ By Table \ref{Ta3},
$S$ has a unipotent character $\psi$ with $\psi(1)_p\geq p^{b(n-1)^2-b}$ and so $bn(n-2)\leq 91a.$
Multiplying both sides by $n,$ we have $bn^2(n-2)=120a(n-2)\leq 91an.$ Thus $29n\leq 240$ and so
$n\leq 8.$ Hence $n=8$ and then $8b=15a.$ We see that $14b=105a/4>26a$ so that $l_{14b}(p)$ exists
and $l_{14b}(p)\in \pi(H)$ as $l_{14b}(p)\in \pi(S).$ It follows that $14b\mid 30a=16b,$ which is
impossible.

$(c)$ $S=O_{2n}^\epsilon(r),n\geq 4.$ We have $bn(n-1)=120a.$ By Table \ref{Ta3}, $S$ has a
unipotent character $\psi$ with $\psi(1)_p\geq p^{b(n-1)(n-2)}.$ Thus $b(n-1)(n-2)\leq 91a.$
Multiplying both sides by $n$ and simplifying, we obtain $29n\leq 240$ and so $4\leq n\leq 8.$ If
$2(n-1)b=6$ then $b=1$ and $n=4$ so that $120a=12$ which is impossible. Thus
$l_{2(n-1)b}(p)\in\pi(S)\subseteq\pi(H)$ exists and hence $2b(n-1)\leq 30a.$ Multiplying both sides
by $n$ and simplifying, we obtain $n\geq 8.$ Therefore $n=8$ and so $7b=15a.$ We have
$12b=180a/7>25a$ so that $l_{12b}(p)\in\pi(S)\subseteq\pi(H)$ exists and then $12b\mid 30a=14b,$
which is impossible.

$(d)$ $S={}^2B_2(r^2),{}^2F_4(r^2)$ or ${}^2G_2(r^2),$ where $r^2=2^{2n+1},2^{2n+1}$ or $3^{2n+1},$
respectively. In these cases, the equation $|S|_p=|H|_p$ cannot happen.

$(e)$ $S={}^3D_4(r).$ Then $12b=120a$ and so $b=10a$ or $r=q^{10}.$
We have $r^{6}-1=q^{60}-1\mid |S|$ so that $l_{60}(q)\in\pi(H),$
which is impossible.

$(f)$ $S={}^2E_6(r)$ or $E_6(r).$ Then $36b=120a$ and so $3b=10a.$
We have $12b=40a$ and $p^{12b}-1=q^{40}-1\mid |S|$ so that
$l_{40}(q)\in\pi(H),$ which is impossible as $40a>30a.$

$(g)$ $S=G_2(r).$ Then $6b=120a$ and so $b=20a.$ We have
$r^6-1=q^{120}-1\mid |S|$ so that $l_{120}(q)\in\pi(H),$ which is
impossible.

$(h)$ $S=F_4(r).$ Then $24b=120a$ and so $b=5a.$ As
$p^{8b}-1=q^{40}-1\mid |S|,$ we deduce that $l_{40}(q)\in\pi(H),$
which is impossible.

$(i)$ $S=E_7(r).$ Then $63b=120a$ and so $21b=40a.$ We have
$18b=240a/7>30a$ so that $l_{18b}(p)$ exists and hence
$l_{18b}(p)\in \pi(H)$ so that $18b\leq 30a,$ a contradiction.

$(j)$ $S=E_8(r).$ Then $120b=120a$ and so $b=a.$ Hence $S\cong H.$

\end{proof}
{\bf Proof of Theorem \ref{main}.} This follows from Propositions
\ref{prop1}-\ref{prop5}. $\hfill\square$

{\bf Proof of Corollary \ref{main1}.} Assume that $G$ is a perfect group satisfying
$\textrm{cd}(G)\subseteq \textrm{cd}(H)$ and $|G|\leq |H|.$ Let $N$ be a maximal normal subgroup of
$G.$ Then $G/N$ is a non-abelian simple group. As $\textrm{cd}(G/N)\subseteq \textrm{cd}(G)$ and
$\textrm{cd}(G)\subseteq \textrm{cd}(H),$ it follows that $\textrm{cd}(G/N)\subseteq
\textrm{cd}(H).$ By Theorem \ref{main}, we obtain $G/N\cong H.$ Hence $|G/N|=|H|.$ Since $|G|\leq
|H|,$ we conclude that $|N|=1.$ Hence $N$ is a trivial subgroup of $G,$ so that $G\cong H.$ The
proof is now complete. $\hfill\square$

{\bf Proof of Corollary \ref{main2}.} Assume that $G$ is a group satisfying
$\textrm{X}_1(G)\subseteq \textrm{X}_1(H).$ It follows that $|G|\leq |H|$ and
$\textrm{cd}(G)\subseteq \textrm{cd}(H).$ Moreover, as $H$ is non-abelian simple, $H$ has a unique
irreducible complex character of degree $1,$ which is the trivial character of $H$ and as
$\textrm{X}_1(G)\subseteq \textrm{X}_1(H),$ $G$ also possesses a unique character of degree $1.$
Thus $G$ must be
 perfect. Now the result follows from Corollary \ref{main1}.
$\hfill\square$

{\bf Proof of Corollary \ref{main3}.} Assume that $\C G\cong \C H.$ By Molien's Theorem (see
\cite[Theorem $2.13$]{Ber1}), the first columns of the ordinary character tables of $G$ and $H$
coincide so that $\textrm{X}_1(G)=\textrm{X}_1(H).$ Hence the result follows from  Corollary
\ref{main2}. $\hfill\square$

\subsection*{Acknowledgment} The author is grateful to the referee
for his or her comments.

\end{document}